\documentclass[12pt,reqno]{amsart}
\usepackage{amsmath}
\usepackage{amssymb}
\usepackage{amstext}
\usepackage{a4wide}
\usepackage{bm}
\usepackage{graphicx}
\usepackage{txfonts}
\usepackage[numbers,sort&compress]{natbib}
\allowdisplaybreaks \numberwithin{equation}{section}
\usepackage{color}
\usepackage{cases}

\numberwithin{equation}{section}

\newtheorem{theorem}{Theorem}[section]
\newtheorem{proposition}[theorem]{Proposition}

\newtheorem{lemma}[theorem]{Lemma}

\theoremstyle{definition}

\newtheorem{definition}[theorem]{Definition}

\theoremstyle{remark}
\newtheorem{remark}[theorem]{Remark}

\begin{document}

\title
[Desingularization of vortices for the incompressible Euler equation on a sphere]{Desingularization of vortices for the incompressible Euler equation on a sphere}

 \author{Daomin Cao, Shuanglong Li, Guodong Wang}
 
 \address{Institute of Applied Mathematics, AMSS, Chinese Academy of Sciences, Beijing 100190, and University of Chinese Academy of Sciences, Beijing 100049,  PR China}
\email{dmcao@amt.ac.cn}

\address{School of Mathematics, Southwestern University of Finance and Economics, Chengdu 611130, PR China }
\email{lishl@swufe.edu.cn}

\address{School of Mathematical Sciences, Dalian University of Technology, Dalian 116024, PR China}
\email{gdw@dlut.edu.cn}

%\thanks{}

\begin{abstract}
In this paper, we construct a family of global  solutions to the incompressible Euler equation on a standard 2-sphere. These solutions are odd-symmetric  with respect to the equatorial plane and rotate with a constant angular speed around the polar axis.   More importantly, these solutions ``converges" to a pair of point vortices with equal strength and opposite signs. The construction is achieved by maximizing the energy-impulse functional relative to a family of suitable rearrangement classes and analyzing the asymptotic behavior of the maximizers. Based on their variational characterization, we also prove the stability of these rotating solutions with respect to odd-symmetric perturbations. 

\end{abstract}

\maketitle

\tableofcontents
\section{Introduction and main results}\label{sec1}

\subsection{Background and notations}

Flows with  concentrated vorticity represent a fundamental concept in fluid dynamics, in which the rotation of the fluid is highly localized within small regions. For a two-dimensional ideal (i.e., incompressible and inviscid) fluid, these flows are often modeled as \emph{point vortices}, where each vortex is mathematically represented by a Dirac measure, denoting the location of a concentrated source of vorticity. The point vortex model idealizes complex fluid behavior, enabling simplified analysis of the interactions, stability, and movement of concentrated vorticity blobs.  Since the pioneering works of Helmholtz \cite{Helm} and Kirchhoff \cite{Kir} in the 19th century, the study of the point vortex model has attracted the attention of numerous mathematicians and physicists.

 It is well known that the evolution of a two-dimensional ideal fluid is governed by the Euler equation. Formally, the point vortex model can be regarded as a singular limit of the Euler equation by neglecting the self-induced singular velocity for each concentrated vorticity blob. An important issue is to investigate the relationship between the Euler equation  and the point vortex model on a rigorous level. More specifically, there are two types of problems:
\begin{itemize}
  \item [(P1)] Suppose that the initial vorticity of a two-dimensional ideal fluid is sharply concentrated around a finite number of points, does the evolved vorticity according to the Euler equation remain concentrated? If so, do the concentration points satisfy the point vortex model?  How to obtain the optimal concentration rate? 
  \item [(P2)] Can we construct as many coherent flow patterns (such as steady/traveling solutions) of the Euler equation as possible such that the vorticity evolves like point vortices? What about the dynamical stability of these flow patterns?
\end{itemize}
In the literature, (P1) and (P2) are often referred to as \emph{desingularization of vortices}. The study of (P1) was initiated by Marchioro and Pulvirenti \cite{MP83}, then followed by many authors \cite{BM18,CM15,CS22,DI,M88,MP83,MP93,T87}. Roughly speaking, the answer to (P1) is positive, although some open problems still remain, such as the optimal  concentration rate.  As to (P2), there have been a large number of results concerning the desingularization of various relative equilibria of point vortices in the literature. See \cite{CLW14,CPY15,CWZ20,EM91,EM94,SV10,T83,W23,W241,W242} and the references therein. There are mainly two methods to tackle (P2): one from the perspective of the stream function, which reduces the problem to solving a singularly perturbed semilinear elliptic equation  \cite{CLW14,CPY15,SV10}, and the other  from the perspective of vorticity, which usually involves solving a variational problem related to vorticity \cite{CWZ20,EM91,EM94,T83,W23,W241,W242}.

The aforementioned works are all conducted in two-dimensional Euclidean domains. It is natural to ask whether similar desingularization results hold on a two-dimensional surface, particularly regarding the difficulties that might arise from the non-trivial metric and the different topology of the surface. To the best of our knowledge, there is very little research in this respect, although the motion of a finite number of point vortices on a closed surface has already been investigated by many authors \cite{Bo77,BK08,Ki87,New01,WQ21}.

The purpose of this paper is to fill the research gap by  studying the desingularization problem (P2) on  a standard two-dimensional sphere. We consider the simplest point vortex equilibrium, i.e.,   a  pair of point vortices placed symmetrically at equal latitudes on either side of the same longitude that rotates with a constant angular speed around the polar axis.   Based on  a variational approach, we construct a family of rotating solutions of the Euler equation such that the vorticity has small support and ``converges" to the  pair of point vortices. Although the sphere is the simplest closed surface, the study of fluid motion on it holds significant value in both theoretical and applied sciences, including geophysical and atmospheric dynamics. 
We believe that this work can inspire and assist in addressing more complicated desingularization problems on general  closed surfaces.

Before presenting the main results, we list some notations that will be frequently used throughout this paper:
 \begin{itemize}
   \item  Denote by  $\mathbb S^2$  the unit 2-sphere in $\mathbb R^3$ centered at the origin,
\[\mathbb S^2=\left\{\mathbf x\in\mathbb R^3\mid x_1^2+x_2^2+x_3^2=1\right\},\]
where $(x_1,x_2,x_3)$ is the Cartesian coordinates of $\mathbf x$ in $\mathbb R^3$. 
The metric on  $\mathbb S^2$ is the standard metric induced by the ambient space $\mathbb{R}^3$. Denote by  $|\cdot|$ the Riemannian measure (i.e., surface area) on   $\mathbb S^2$ associated with the standard metric. 
\item
 The spherical coordinates $(\varphi,\theta)$ of $\mathbb S^2$ are 
\[x_1=\cos\theta\cos\varphi,\quad x_2=\cos\theta\sin\varphi,\quad x_3=\sin\theta,\quad  (\varphi,\theta)\in(0,2\pi)\times \left(-\frac{\pi}{2},\frac{\pi}{2}\right).\]
The Riemannian volume element in spherical coordinates $(\varphi,\theta)$ is
\[d\sigma=\cos\theta d\varphi d\theta.\]
 \item For $\mathbf x\in \mathbb S^2$, denote by $T_{\mathbf x}\mathbb S^2$ the tangent plane of $\mathbb S^2$ at $\mathbf x$. For $\mathbf p\in T_{\mathbf x}\mathbb S^2$, denote by $J\mathbf p$ the  \emph{clockwise} rotation through $\pi/2$ of $\mathbf p$ in $T_{\mathbf x}\mathbb S^2$, i.e.,
    \[J\mathbf p=\mathbf p\times \mathbf x.\]
\item Denote 
\[ N:=(0,0,1),\quad S:=(0,0,-1),\]
and
\[\mathbb S^2_+:=\left\{\mathbf x=(x_1,x_2,x_3)\in\mathbb S^2\mid x_3>0\right\},\quad \mathbb S^2_-:=\left\{\mathbf x=(x_1,x_2,x_3)\in\mathbb S^2\mid x_3<0\right\}.\]
\item For $\mathbf x,\mathbf y\in\mathbb S^2$, denote by $d(\mathbf x,\mathbf y)$ the geodesic distance 
between $\mathbf x$ and $\mathbf y$. It is easy to check that
\begin{equation}\label{gds}
|\mathbf x-\mathbf y|=2\sin\left(\frac{d(\mathbf x,\mathbf y)}{2}\right).
\end{equation}
Denote by $B_r(\mathbf x)$ the geodesic disk centered at $\mathbf x$ with geodesic radius $r$, i.e.,
\[B_r(\mathbf x):=\left\{\mathbf y\in\mathbb S^2\mid d(\mathbf x,\mathbf y)<r\right\}.\]
The surface area of $B_r(\mathbf x)$ can be computed as follows:
\begin{equation}\label{sagd}
|B_r(\mathbf x)|=|B_r(N)|=\int_0^{2\pi}\int_{\pi/2-r}^{\pi/2}\cos\theta d\theta d \varphi=2\pi(1-\cos r).
\end{equation}
\item  Let  $\Delta$ be the Laplace-Beltrami operator on $\mathbb S^2$ with respect to the standard metric.  The Green function  of $-\Delta$ on $\mathbb S^2$ (which is unique up to a constant) is
 \begin{equation}\label{gfs1}
 G(\mathbf x,\mathbf y)=-\frac{1}{2\pi}\ln|\mathbf x-\mathbf y|,
 \end{equation}
 where $|\mathbf x-\mathbf y|$ is the distance between $\mathbf x$ and $\mathbf y$ in $\mathbb R^3$. See  \cite{Au,CG}. For $v:\mathbb S^2\mapsto \mathbb R,$ define  
  \begin{equation}\label{ine01}
 \mathcal Gv(\mathbf x):=\int_{\mathbb S^2}G(\mathbf x,\mathbf y)v(\mathbf y)d\sigma=-\frac{1}{2\pi}\int_{\mathbb S^2}\ln|\mathbf x-\mathbf y|v(\mathbf y)d\sigma, 
 \end{equation}
 Then $\mathcal Gv$ satisfies
  \begin{equation}\label{ine02}
  -\Delta \mathcal Gv=v-\frac{1}{|\mathbb S^2|}\int_{\mathbb S^2}vd\sigma\,\,\,\,\mbox{on}\,\,\,\,\mathbb S^2.
  \end{equation}
If  additionally $\int_{\mathbb S^2}vd\sigma=0,$ then 
    \begin{equation}\label{ine021}
  -\Delta \mathcal Gv=v \,\,\,\,\mbox{on}\,\,\,\,\mathbb S^2,\quad \int_{\mathbb S^2}\mathcal Gv d\sigma=0.
  \end{equation}
 \item
 The Green function  of $-\Delta$ on $\mathbb S^2_+$  with zero boundary condition is
 \begin{equation}\label{gfs2}
 G_+(\mathbf x,\mathbf y)= \frac{1}{2\pi}\ln\frac{|\mathbf x-  \mathbf y'|}{|\mathbf x-\mathbf y|},
 \end{equation}
 where $\mathbf y'$ is the reflection point of $\mathbf y$ with respect to the equatorial plane, i.e., 
 \begin{equation}\label{rep1}
 \mathbf y'=(y_1,y_2,-y_3)\,\, \mbox{ for }\,\, \mathbf y=(y_1,y_2,y_3).
 \end{equation}
For $v:\mathbb S_+^2\mapsto \mathbb R,$ define 
  \begin{equation}\label{ine03}
\mathcal G_+v(\mathbf x):=\int_{\mathbb S_+^2}G_+(\mathbf x,\mathbf y)v(\mathbf y)d\sigma= \frac{1}{2\pi}\int_{\mathbb S_+^2}\ln\frac{|\mathbf x-\mathbf y'|}{|\mathbf x-\mathbf y|}v(\mathbf y)d\sigma.
 \end{equation}
Then $\mathcal G_+v$ satisfies
   \begin{equation}\label{ine04}
   -\Delta \mathcal G_+v=v\,\,\,\,\mbox{on}\,\,\,\,\mathbb S_+^2,\quad   \mathcal G_+v\big|_{\partial \mathbb S^2_+}=0.
   \end{equation}
   \item For $\alpha\in\mathbb R$ and some \emph{unit} vector $\mathbf p\in\mathbb R^3$, denote by $\mathsf R^\mathbf p_\alpha$   the $\alpha$-degree rotation around the vector $\mathbf p$. By  Rodrigues' rotation formula,  
\begin{equation}\label{rod1}
\mathsf R^{\mathbf p}_\alpha\mathbf x=(\cos\alpha)\mathbf x+\sin\alpha(\mathbf p\times \mathbf x)+(1-\cos\alpha)(\mathbf p\cdot\mathbf x)\mathbf p\quad\forall\,\mathbf x\in\mathbb R^3.
\end{equation}

   \item For any measurable (with respect to the Riemannian measure) function $v:\mathbb S_+^2\mapsto \mathbb R,$  denote by $\mathcal R_v$ the rearrangement class of $v$ on $\mathbb S^2_+$, i.e.,
       \[\mathcal R_v:=\left\{w:\mathbb S^2_+\mapsto\mathbb R\mid |\{\mathbf x\in\mathbb S^2_+\mid w(\mathbf x)>s\}|= |\{\mathbf x\in\mathbb S^2_+\mid v(\mathbf x)>s\}|\,\,\forall\,s\in\mathbb R\right\}.\]
 \end{itemize}

  \subsection{Euler equation on a sphere}\label{sec11}
  
Consider an ideal
fluid of unit density on $\mathbb S^2$. In terms of the velocity $\mathbf v(t,\mathbf x)\in T_{\mathbf x}\mathbb S^2$ and the scalar pressure $P(t,\mathbf x)$, the motion of the fluid can be described by the following incompressible Euler equation:
 \begin{equation}\label{eupf}
  \begin{cases}
    \partial_t\mathbf v+\nabla_{\mathbf v}\mathbf v= -\nabla P, &  t\in\mathbb R, \,\,\mathbf x\in\mathbb S^2, \\
    {\rm div}\, \mathbf v=0.
  \end{cases}
\end{equation}
where $\nabla_{\mathbf v}$ is the  covariant derivative along $\mathbf v$,
$\nabla$ is the gradient operator, and ${\rm div}$ is the divergence operator, all with respect to the standard metric on $\mathbb S^2$. Similar to the case of the Euler equation in two-dimensional Euclidean domains, global well-posedness of \eqref{eupf} in suitable function spaces can be proved by using the classical methods in \cite{MB,MP94}. For instance, a detailed proof of global well-posedness of \eqref{eupf}  in $H^s$ for any $s>2$  can be found in \cite{Ta16}.
Note that one can also consider a rotating sphere around some axis, which is physically meaningful (cf. \cite{CG}). However, the main results of this paper are not closely related to the rotation  effect  of the sphere, so we only provide a brief discussion on the rotating  case in Section \ref{sec6}.

Below we derive the vorticity formulation of \eqref{eupf}. Since $\mathbf v$ is divergence-free, there exists a scalar function $\psi$, called the \emph{stream function},  such that 
\[\mathbf v:=J\nabla\psi.\]
 The  \emph{vorticity} (also referred to as the  \emph{relative vorticity}) $\omega$ is  defined by 
\[\omega:={\rm div}(J\mathbf v).\]
Note that $\omega$ necessarily satisfies the Gauss constraint
 \begin{equation}\label{s01}
 \int_{\mathbb S^2}\omega d\sigma =0.
\end{equation}
It is easy to check that   $\psi$ and  $\omega$ satisfy the  Poisson equation
 \begin{equation}\label{s02}
 -\Delta\psi =\omega.
 \end{equation}
By adding a suitable constant to $\psi$, we can assume that
 \begin{equation}\label{s03}
 \int_{\mathbb S^2}\psi d\sigma=0.
 \end{equation}
 Using the Green operator $\mathcal G$ (cf. \eqref{ine01})  and taking into account   \eqref{ine021} and \eqref{s01}, we have that
\[
 \psi  =\mathcal G \omega,\quad \mathcal G \omega(t,\mathbf x) :=-\frac{1}{2\pi}\int_{\mathbb S^2}\ln|\mathbf x-\mathbf y|\omega(t,\mathbf x)d\sigma.
\] 
On the other hand, applying the operator ${\rm div}J$ to the first equation of \eqref{eupf} yields
\[\partial_t\omega+\mathbf v\cdot\nabla\omega=0.\]
To summarize,  the  Euler equation \eqref{eupf} in terms of the vorticity $\omega$ can be written as
\begin{equation}\label{vor0}
\begin{cases}
\partial_t\omega+  J\nabla \mathcal G\omega \cdot\nabla \omega=0,\\
\mathcal G\omega (t,\mathbf x)=-\frac{1}{2\pi}\int_{\mathbb S^2}\ln|\mathbf x-\mathbf y|\omega(t,\mathbf x)d\sigma.
\end{cases}
\end{equation}

An  odd-symmetric solution $\omega(t,\mathbf x)$ of \eqref{vor0} is a solution  such that
\[\omega(t,\mathbf x)=-\omega(t, \mathbf x')\quad \forall\,\mathbf x\in\mathbb S^2,\]
 where $ \mathbf x'$ is the reflection point of $\mathbf x$ with respect to the  equatorial plane (cf. \eqref{rep1}).
 In the  odd-symmetric setting, it suffices to consider the following Euler equation on the northern hemisphere $\mathbb S^2_+$:
 \begin{equation}\label{eun0}
 \begin{cases}
 \partial_t\omega +J\nabla \mathcal G_+\omega \cdot\nabla \omega=0, \quad t\in\mathbb R,\,\mathbf x\in \mathbb S^2_+,\\
\mathcal G_+\omega (t,\mathbf x)= \frac{1}{2\pi}\int_{\mathbb S^2}\ln\frac{|\mathbf x-\mathbf y'|}{|\mathbf x-\mathbf y|}\omega(t,\mathbf x)d\sigma.
 \end{cases}
 \end{equation}
 In fact, it is easy to check, at least formally, that
 \begin{itemize}
   \item [(1)] if $\omega$ is an odd-symmetric solution of  \eqref{vor0},  then the restriction of $\omega$ on $\mathbb S^2_+$  solves \eqref{eun0};
   \item [(2)]if $\omega $ solves \eqref{eun0}, then
   \[\tilde \omega:=\begin{cases}
              \omega (\mathbf x), & \mbox{if } \mathbf x\in \mathbb S^2_+, \\
             -\omega ( \mathbf x'), & \mbox{if } \mathbf x\in \mathbb S^2_-  \\
            \end{cases}\]
      solves \eqref{vor0}.
 \end{itemize}
In the rest of this paper,   we will mainly be focusing on odd-symmetric solutions, so it suffices to consider the vorticity equation  \eqref{eun0}.

Conservation laws  play an important role in this paper. For \eqref{vor0}, there are three conserved quantities: the kinetic energy, the projection of vorticity onto the space of degree-1 spherical harmonics, and the distribution function of vorticity (cf. Proposition 1 of \cite{CG}).
Accordingly, 
for any sufficiently regular solution $\omega(t,\mathbf x)$ of \eqref{eun0},   the following three conservation laws hold:
\begin{equation}\label{cl01}
  E(\omega(t,\cdot))=E(\omega(0,\cdot))\quad\forall\,t\in\mathbb R,\quad E(\omega):= \frac{1}{2}\int_{\mathbb S_+^2}\omega\mathcal G_+\omega d\sigma,
\end{equation}
\begin{equation}\label{cl02}
  I(\omega(t,\cdot))=I(\omega(0,\cdot))\quad\forall\,t\in\mathbb R,\quad I(\omega):=  \int_{\mathbb S_+^2}x_3\omega d\sigma,
\end{equation}
\begin{equation}\label{cl03}
 \omega(t,\cdot)\in\mathcal R_{\omega(0,\cdot)}\quad\forall\,t\in\mathbb R,
\end{equation}
where $E(\omega)$ is the kinetic energy, and $I(\omega)$  can be regarded as the impulse (linear momentum parallel to the equator) of the fluid.

A rotating solution  around the polar axis $\mathbf e_3=(0,0,1)$ of \eqref{eun0} is a solution of the form
\[\omega(t,\mathbf x)=v(\mathsf R^{\mathbf e_3}_{-\lambda t}\mathbf x),\]
where $v:\mathbb S^2_+\mapsto \mathbb R$, and $\lambda$ is a constant representing the angular speed of the rotation.
Using \eqref{rod1}, one can verify that $v(\mathsf R^{\mathbf e_3}_{-\lambda t}\mathbf x)$  solves \eqref{eun0} if and only if $v$ satisfies
\begin{equation}\label{xz01}
J\nabla(\mathcal G_+v-\lambda x_3)\cdot\nabla v=0.
\end{equation}
The weak formulation of \eqref{xz01} is that
   \begin{equation}\label{wsf1}
    \int_{\mathbb S^2_+} vJ\nabla (\mathcal G_+v-\lambda x_3)\cdot\nabla\xi d\sigma=0\quad\forall\,\xi\in C_c^\infty(\mathbb S^2_+).
    \end{equation}
 Note that by the regularity theory for elliptic equations (cf. Lemma 3.1 of \cite{CWZ23}) and Sobolev embedding (cf.  \cite{Au,Hebey}), the integral in \eqref{wsf1} makes sense when   $v\in L^{4/3}(\mathbb S^2_+)$.

 \subsection{Point vortices on a sphere}
For $N$ point vortices on $\mathbb S^2$  located at $\mathbf x_1,\cdot\cdot\cdot,\mathbf x_N$, the velocity generated by the $i$-th vortex is
  \[\mathbf v_i:=\kappa_iJ\nabla_\mathbf x G(\mathbf x,\mathbf x_i), \]
   where $G$ is the Green function given by \eqref{gfs1}. 
  Neglecting the self-interaction for each vortex, which is reasonable by symmetry,  we get 
 the following system of ODEs:
 \begin{equation}\label{odes1}
 \frac{d\mathbf x_i }{dt}=\sum_{j\neq i}\kappa_jJ\nabla_\mathbf x G(\mathbf x,\mathbf x_j)\bigg|_{\mathbf x=\mathbf x_i},\quad i=1,\cdot\cdot\cdot,N.
 \end{equation}
By a simple calculation,   \eqref{odes1} can be written as 
 \begin{equation}\label{odes2}
 \frac{d\mathbf x_i }{dt}=\sum_{j\neq i}\frac{\kappa_j}{2\pi}\frac{\mathbf x_j\times\mathbf x_i}{|\mathbf x_i-\mathbf x_j|^2},\quad i=1,\cdot\cdot\cdot,N.
 \end{equation}
This is the point vortex model on $\mathbb S^2$, which was  first  obtained by Bogomolov \cite{Bo77}.  Note that \eqref{odes2} is  a Hamiltonian system with the Hamiltonian $\mathcal H$ given by
 \[\mathcal H(\mathbf x_1,\cdot\cdot\cdot,\mathbf x_N)=-\frac{1}{2\pi}\sum_{i<j}\kappa_i\kappa_j\ln|\mathbf x_i-\mathbf x_j|.\]
 Apart from  $\mathcal H,$ another  conserved quantity of \eqref{odes2} is the following vector:
 \[\mathbf M(\mathbf x_1,\cdot\cdot\cdot,\mathbf x_N)=\sum_{i=1}^N\kappa_i\mathbf x_i.\]
  See Chapter 4 of \cite{New01}.
 
For a 2-vortex system, i.e, $N=2,$  \eqref{odes2} becomes
  \begin{equation}\label{odes3}
 \frac{d\mathbf x_1 }{dt}= \frac{\kappa_2}{2\pi}\frac{\mathbf x_2\times\mathbf x_1}{|\mathbf x_1-\mathbf x_2|^2},\quad
  \frac{d\mathbf x_2 }{dt}= \frac{\kappa_1}{2\pi}\frac{\mathbf x_1\times\mathbf x_2}{|\mathbf x_1-\mathbf x_2|^2}.
 \end{equation}
If  $\kappa_1=-\kappa_2=\kappa$ and
 \[\begin{cases}
 \mathbf x_1(0)=\left(\cos\theta_0 \cos\varphi_0,
\cos\theta_0\sin\varphi_0,\sin\theta_0\right),\\
  \mathbf x_2(0)=\left(\cos\theta_0\cos\varphi_0,
\cos\theta_0\sin\varphi_0,-\sin \theta_0\right),
 \end{cases}\]
 where $\theta_0, \varphi_0\in\mathbb R$, then the solution to \eqref{odes3} is given by
 \begin{equation}\label{2vm}
 \begin{cases}
 \mathbf x_1(t)=\left(\cos\theta_0 \cos (\lambda t+\varphi_0),\cos\theta_0 \sin(\lambda  t+\varphi_0),\sin\theta_0\right),\\
  \mathbf x_2(t)=\left(\cos\theta_0 \cos(\lambda t+\varphi_0),
\cos\theta_0 \sin(\lambda t+\varphi_0),-\sin\theta_0\right),
  \end{cases}
  \end{equation}
where $\lambda$ satisfies
 \begin{equation}\label{Are}
4\pi\lambda\sin\theta_0=\kappa.
 \end{equation}
This can be easily verified by a straightforward computation.
Intuitively, the explicit solution \eqref{2vm} means that a pair of point vortices that are odd-symmetric with respect to the equatorial plane would rotate around the polar axis at a constant angular speed.
It is worth mentioning that the relation \eqref{Are} is very similar to that for a pair of counter-rotating vortices in the plane \cite{W241,W242}.

 \subsection{Main results}

Our first result shows the existence of an odd-symmetric rotating  solution to \eqref{eun0} with fixed angular speed and prescribed distribution function. The precise statement is as follows.
 
 \begin{theorem}[Existence]\label{thme}
Let $1<p<\infty$ be fixed, and $\mathcal R$ be the rearrangement class of some $L^p$ function on $\mathbb S_+^2$.  For fixed $\lambda\in\mathbb R$, consider the following maximization problem:
  \[M :=\sup_{v\in\mathcal R}\mathcal E(v),\quad \mathcal E:=E-\lambda I.\tag{$V$}\]
  Denote by $\mathcal M$ the set of maximizers of $(V)$. Then
 \begin{itemize}
 \item[(i)]$\mathcal M\neq\varnothing$.
 \item [(ii)] For any  $v\in\mathcal M$, there exists some  increasing function $\phi_v:\mathbb R\mapsto \mathbb R\cup\{\pm\infty\}$ such that
\[ v=\phi_v(\mathcal G_+ v-\lambda x_3)\quad \mbox{a.e. on }\mathbb S_+^2.\]
\item[(iii)] If $p\geq 4/3$, then for any $v\in\mathcal M$,  $v(\mathsf R^{\mathbf e_3}_{-\lambda t}\mathbf x)$ solves \eqref{eun0} in the weak sense, i.e.,  $v$ satisfies \eqref{wsf1}.
    \end{itemize}
\end{theorem}

\begin{remark}
The increasing function $\phi_v$ in Theorem \ref{thme}(ii) is usually unknown. However, if we take $\mathcal R$ as the following patch class
\[\mathcal R=\left\{\gamma\chi_{A}\mid A\subset\mathbb S^2_+\mbox{ is measurable},\,|A|=a\right\},\]
where $\gamma, a$ are given positive numbers, and $\chi_A$ denotes the characteristic function of $A$, i.e.,
\begin{equation}\label{chi1}
\chi_A(\mathbf x)=
\begin{cases}
1,&\mathbf x\in A,\\
0,&\mbox{otherwise.}
\end{cases}
\end{equation}
then each $\phi_v$ must be a Heaviside-type function, i.e., there exists some $\mu_v\in\mathbb R$ such that 
\[\phi_v(s)=
\begin{cases}
\gamma,&\mbox{if }\,\,s>\mu_v,\\
0,&\mbox{if }\,\,s\leq \mu_v.
\end{cases}\]
This can be easily verified using Lemma \ref{sce01} in Section \ref{sec3}.  
\end{remark}
 The maximization problem $(V)$ is mainly inspired by the variational principle for planar vortex pairs  \cite{T83,BNL,W242}. Existence of a maximizer and its profile rely on Burton's rearrangement theory \cite{BMA}.  The proof of Theorem \ref{thme} will be provided in Section \ref{sec3}.

Notice that the quantities involved in $(V)$ are all conserved quantities of the Euler equation (cf. \eqref{cl01}-\eqref{cl03}). This variational nature  enables us to prove certain stability  of the maximizers, which is our second result.
 \begin{theorem}[Stability]\label{thmos}
In the setting of Theorem \ref{thme}, the set of maximizers $\mathcal M$ is stable under the dynamics of \eqref{eun0} in the following sense:
  for any $\epsilon>0$, there exists some $\delta>0$, such that for any  sufficiently smooth solution  $\omega(t,\mathbf x)$ of \eqref{eun0}, 
 it holds that
  \[\inf_{v\in\mathcal M}\|\omega(0,\cdot)-v\|_{L^p(\mathbb S^2_+)}<\delta\quad\Longrightarrow\quad \inf_{v\in\mathcal M}\|\omega(t,\cdot)-v\|_{L^p(\mathbb S^2_+)}<\epsilon\quad\forall\,t\in\mathbb R.\]
\end{theorem}

The proof of Theorem \ref{thmos} is based on a compactness argument which was first used by Burton \cite{BAR}. In this process, Burton's rearrangement theory \cite{BMA} and conservation laws \eqref{cl01}-\eqref{cl03} of the Euler equation play an essential role.

\begin{remark}
The stability  in Theorem \ref{thmos} still holds if the perturbed solution $\omega(t,\mathbf x)$  belongs to $L^p(\mathbb S^2_+)$ for any $t\in\mathbb R$ and  satisfies \eqref{cl01}-\eqref{cl03}. See Section \ref{sec4}.
\end{remark}

\begin{remark}
 Since both $E$ and $I$ are invariant under rotations around the polar axis, i.e.,
\[E(v)=E(v\circ\mathsf R^{\mathbf e_3}_\alpha),\quad I(v)=I(v\circ\mathsf R^{\mathbf e_3}_\alpha) \quad\forall\,\alpha\in\mathbb R,\]
we deduce that the set $\mathcal M$ is closed with respect to rotations around the polar axis, i.e.,
\[v\in\mathcal M\quad\Longrightarrow \quad v\circ\mathsf R^{\mathbf e_3}_\alpha\in\mathcal M\quad\forall\,\alpha\in \mathbb R.\]
However, more details about the structure of $\mathcal M$ remain unknown.
We conjecture that, at least for some special rearrangement classes,  the maximizer of $(V)$ is unique up to rotations around the polar axis, i.e., there exists some $\hat v\in\mathcal M$ such that 
\[\mathcal M=\left\{\hat v\circ\mathsf R^{\mathbf e_3}_\alpha\mid \alpha\in\mathbb R\right\}.\]
\end{remark}

\begin{remark}
By now, it is not clear whether the elements in $\mathcal M$ are zonal (i.e., rotationally invariant under rotations around the polar axis).  If this happens, then the rotating solutions we have obtained are actually trivial. In the following theorem, we will show that  for some special rearrangement classes and for some special values of $\lambda$, any maximizer cannot be zonal.
\end{remark}

To desingularize the vortex pair \eqref{2vm}, we choose  a family of appropriate  rearrangement classes in $(V)$ and study  the asymptotic behavior of the maximizers. 
For $0<\varepsilon\leq \pi/2$, consider a family of  functions $\{\varrho_\varepsilon\}\subset L^p(\mathbb S^2_+)$  satisfying the following three hypotheses:
 \begin{itemize}
   \item [(H1)]  For every  $\varepsilon\in(0, {\pi}/{2}],$ $\varrho_\varepsilon$ is nonnegative, and
       \[   \varrho_\varepsilon(\mathbf x)=0\quad\mbox{a.e. }  \,\mathbf x\in\mathbb S^2_+\setminus B_\varepsilon(N).\]
   \item [(H2)]There exists some $\kappa>0$ such that
   \[\int_{\mathbb S^2_+}\varrho_\varepsilon d\sigma=\kappa \quad \forall\,\varepsilon\in(0, {\pi}/{2}].\] 
   \item[(H3)]There exists some $K>0$ such that
   \[\varepsilon^{2/p'}\|\varrho_{\varepsilon}\|_{L^p(\mathbb S^2_+)}\leq K\quad \forall\,\varepsilon\in(0, {\pi}/{2}].\]
 \end{itemize}
As a simple example, we can take  $\varrho_\varepsilon$ as the following patch function:
 \begin{equation}\label{vpss}
 \varrho_{\varepsilon}(\mathbf x)=\frac{\kappa}{2\pi(1-\cos \varepsilon)}\chi_{B_\varepsilon(N)},
 \end{equation}
 where $\chi_{B_\varepsilon(N)}$ is the characteristic function of $B_\varepsilon(N)$  (cf. \eqref{chi1}).
Note that such $\varrho_\varepsilon$ satisfies (H2) by \eqref{sagd}, and satisfies (H3) by H\"older's inequality.

Our third result is as follows.

\begin{theorem}[Limiting behavior]\label{thmlb}
Let $1<p<\infty$ and $0<\varepsilon\leq \pi/2$ be fixed. Suppose that $\{\varrho_\varepsilon\}\subset L^p(\mathbb S^2_+)$  satisfies (H1)-(H3), where $\kappa,K>0$ are fixed.
Let $\mathcal R_\varepsilon$ be the rearrangement class of $\varrho_\varepsilon$.  
For fixed $\lambda>0$, consider the following maximization problem:
  \[M_\varepsilon :=\sup_{v\in\mathcal R_\varepsilon}\mathcal E(v),\quad \mathcal E:=E-\lambda I.\tag{$V_\varepsilon$}\]
 Denote by $\mathcal M_\varepsilon $ the set of maximizers of $(V_\varepsilon)$  (which is non-empty by Theorem \ref{thme}).  Then 
 
 \begin{itemize}
 \item [(i)]   For any $v\in\mathcal M_\varepsilon$, there exists some open set $O_v\subset\mathbb S^2_+$, called the vortex core related to $v$, such that 
     \begin{equation}\label{vc}
       v(\mathbf x)
   >0 \,\,\,\, \mbox{a.e.}\,\,\,\, \mathbf x\in O_v,\,\,\,\,  v(\mathbf x)
   =0 \,\,\,\, \mbox{a.e.}\,\,\,\, \mathbf x\in \mathbb S^2_+\setminus O_v.
     \end{equation}
 \item [(ii)]  There exists some $C>0$,  not depending on $\varepsilon$, such that
\[O_v\subset  B_{C\varepsilon}(\mathbf X_v)\quad\forall\,v\in\mathcal M_\varepsilon,\]
where $\mathbf X_v$ is the mass center of $v$, 
\begin{equation}\label{doms}
\mathbf X_v:=\frac{1}{\kappa}\int_{\mathbb S^2_+}\mathbf xv d\sigma.
\end{equation}
\item[(iii)]  
For arbitrarily chosen $ v_\varepsilon\in\mathcal M_\varepsilon$, suppose up to a subsequence that
\[\lim_{\varepsilon\to 0}\mathbf X_{v_\varepsilon}= \hat{\mathbf x}.\]
 Then the third component $\hat x_3$ of $\hat{\mathbf x}$ satisfies
  \begin{equation}\label{llx3}
  \hat x_3=\min\left\{1, \frac{\kappa}{4\pi\lambda}\right\}.
 \end{equation}
  \end{itemize}
\end{theorem}

 Note that  \eqref{llx3}  is exactly \eqref{Are} when $\kappa/(4\pi\lambda)\leq  1$.
According to Theorem \ref{thmlb}, when $\kappa/(4\pi\lambda)< 1$, the limiting location  of the vortex cores cannot be the north pole $N$, thus  any  $v\in\mathcal M_\varepsilon$ must be non-zonal if $\varepsilon$ is sufficiently small. Hence Theorems \ref{thme}, \ref{thmos} and \ref{thmlb} together confirm the existence of a large class of nontrivial, stable rotating solutions of the Euler equation with sharply concentrated vorticity. In particular, if we take $\varrho_\varepsilon$ as in \eqref{vpss}, then we get a family of  nontrivial, stable, concentrated rotating vortex patches.

The proof of Theorem \ref{thmlb} relies on an energy analysis method established by Turkington \cite{T83}. See also \cite{EM91,EM94}. Since the sphere has very different geometry and topology from planar domains, new challenges arise in the details of the proof. A fortunate fact is that  the Green  function of the sphere has the same 
  singularity as those of planar domains (both exhibiting logarithmic singularity), which allows us to obtain  some key energy estimates as in the planar case.
The detailed proof is provided in Section \ref{sec5}.

Finally, we note that the limiting profile of the maximizers remains unclear. We conjecture that, after a suitable scaling, the limit of the maximizers is a radially decreasing function with respect to some point, as in the planar case \cite{W241,W242}. However, we are unable to prove this. The main difficulties are that  scaling a function on the sphere is more complicated, and the Riesz rearrangement inequality on the sphere is also unknown.
 
%An interesting problem is to desingularize the pair of rotating vortices \eqref{2vm}  outside the odd-symmetric setting. It may be helpful to consider the maximization of the kinetic energy relative to some rearrangement class on $\mathbb S^2$ with prescribed  impulse, as in \cite{BPR,BJD,W242} for planar vortex pairs.

This paper is organized as follows. In Section \ref{sec2}, we present some preliminary materials that will be used in subsequent sections.  
Sections \ref{sec3}-\ref{sec5} are devoted to  the proofs of Theorem \ref{thme}, \ref{thmos} and \ref{thmlb}, respectively. In Section \ref{sec6}, we briefly discuss the case of a rotating sphere.

\section{Preliminaries}\label{sec2}

From now on, let $1<p<\infty$ be fixed, and denote by $p'=p/(p-1)$ the H\"older conjugate of $p$.
We begin with some basic properties of the Green operator $\mathcal G_+$ on $\mathbb S^2_+$ that will be frequently used later.
\begin{lemma}[\cite{CWZ23}, Section 3]\label{greo}
Let $\mathcal G_+$ be the Green operator of $-\Delta$ on $\mathbb S_+^2$ with zero boundary condition (cf. \eqref{ine03} in Section \ref{sec1}). Then
\begin{itemize}
\item[(i)] $\mathcal G_+$ is a bounded linear operator from $L^p(\mathbb S_+^2)$ onto $W^{2,p}(\mathbb S_+^2)$;
   \item[(ii)]  $\mathcal G_+$ is  a compact linear operator from $L^p(\mathbb S_+^2)$ into $L^r(\mathbb S_+^2)$ for any $1\leq r\leq \infty$;
    \item[(iii)]  $\mathcal G_+$ is symmetric, i.e.,
        \[
        \int_{\mathbb S_+^2}v\mathcal G_+wd\sigma=\int_{\mathbb S^2}w\mathcal G_+vd\sigma\quad\forall\, v,w\in  L^p(\mathbb S_+^2);
        \]
        \item[(iv)]   $\mathcal G_+$ is  positive-definite, i.e.,
     \[
          \int_{\mathbb S^2}v\mathcal G_+vd\sigma\geq 0  \quad\forall\,v\in   L^p(\mathbb S_+^2),
        \]
        and the equality  holds if and only if $v\equiv 0.$
\end{itemize}
\end{lemma}

The following technical lemma will be used in proving Theorem \ref{thme}(iv).
\begin{lemma}[\cite{CW20}, Lemma 2.1]\label{lem200}
Suppose that $\phi:\mathbb R\to\mathbb R\cup\{\pm\infty\}$ is a monotone function. Then there exists  a sequence of bounded, smooth functions $\{\phi_n\}_{n=1}^\infty$ such that
\[|\phi_n(s)|\leq |\phi(s)|\quad \forall\,s\in \mathbb R,\,n=1,2,\cdot\cdot\cdot,\]
\[\lim_{n\to\infty}\phi_n(s)=\phi(s)\quad \mbox{whenever }\phi \mbox{ is continuous at }s.\]

\end{lemma}

 The following three lemmas  concerning variational properties of rearrangement classes  will be used in proving both existence and stability.

\begin{lemma}[\cite{BMA}, Theorem 4]\label{lem201}
 Let $\mathcal R_1,$ $\mathcal R_2$ be two rearrangement classes  of some $L^p$ function and some $L^{p'}$ function on $\mathbb S^2_+$, respectively. Then for any $g_2\in\mathcal R_2,$ there exists $g_1\in \mathcal R_1$, such that
\[\int_{\mathbb S^2_+} g_1g_2 d\sigma\geq \int_{\mathbb S^2_+}f_1f_2d\sigma \quad\forall\, f_1\in {\mathcal R}_1,\,\,f_2\in {\mathcal R}_2.\]
\end{lemma}
  \begin{lemma}[\cite{BMA}, Theorem 5]\label{lem202}
 Let $\mathcal R$ be the rearrangement class of some $L^p$ function on $\mathbb S^2_+$. Let $g\in L^{p'}(\mathbb S^2_+)$ be fixed. Define $L: L^p(\mathbb S^2_+)\mapsto\mathbb R$ by setting
 \[L(f)=\int_{\mathbb S^2_+}fg d\sigma,\quad g\in L^s(\mathbb S^2_+).\]
 If $f^*$ is the unique maximizer of $L$ relative to $\mathcal R,$ then  there exists some increasing function $\phi:\mathbb R\mapsto\mathbb R\cup\{\pm\infty\}$ such that
       \[f^*(\mathbf x)=\phi(g(\mathbf x))\quad \mbox{ a.e. }\mathbf x\in\mathbb S^2_+.\]
 \end{lemma}

\begin{lemma}[\cite{BACT}, Lemma 2.3]\label{lem203}
Let  $\mathcal R_1$ and $\mathcal R_2$ be two  rearrangement classes  of two $L^{p}$ functions  on $\mathbb S^2_+$. Then for any $g_1\in\mathcal R_1$, there exists $g_2\in \mathcal R_2$ such that
\begin{equation}\label{kd01}
\|g_1-g_2\|_{L^p(\mathbb S^2_+)}=\inf\left\{\|f_1-f_2\|_{L^p(\mathbb S^2_+)}\mid  f_1\in\mathcal R_1,\, f_2\in\mathcal R_2 \right\}.
\end{equation}
\end{lemma}

We end this section with the following technical lemma that will be used in Section \ref{sec5}.

 \begin{lemma}\label{tcl}
  Suppose that $u\in W^{2,p} (\mathbb S^2_+)$ satisfies
$ u<0$ on $\partial \mathbb S^2_+$.
  Denote
\[  U:= \{\mathbf x\in\mathbb S^2_+\mid u(\mathbf x)>0\}.\]
 Then there exists some $C>0$, depending only on $p$, such that
\[\|\nabla u\|_{L^{2}(U)}\leq C\|\Delta u\|_{L^{p}(U)} |U|^{1/p'}.\]
 \end{lemma}

\begin{proof}
By the Sobolev embedding theorems for compact manifolds (cf. \cite{Au,Hebey}), we have that  $u\in W^{1,r}(\mathbb S^2_+)$ for some $r>2$. In particular, $u\in L^\infty(\mathbb S^2_+).$ 
Denote by $u^+$   the positive part of $u$, i.e.,
\[u^+(\mathbf x):=
\begin{cases}
u(\mathbf x),&\mbox{if }\,\,u(\mathbf x)>0,\\
0,&\mbox{if }\,\,u(\mathbf x)\leq 0.
\end{cases}\]
Since $u<0$ on $\partial \mathbb S^2_+$, it holds  that $u^+\in W_0^{1,r}(\mathbb S_+^2)$  for some $r>2$,  and  
\[\nabla u^+(\mathbf x):=
\begin{cases}
\nabla u(\mathbf x),&\mbox{if }\,\,u(\mathbf x)>0,\\
\mathbf 0,&\mbox{if }\,\,u(\mathbf x)\leq 0.
\end{cases}\]
By integration by parts and H\"older's inequality, 
\begin{equation}\label{e201}
\int_{U}|\nabla u|^{2}d\sigma=\int_{\mathbb S^2_+}  \nabla u^+\cdot\nabla ud\sigma=\int_{\mathbb S^2_+}(-\Delta u)u^+ d\sigma=\int_{U}(-\Delta u)u dx\leq \|u\|_{L^{p'}(U)}\|\Delta u\|_{L^{p}(U)}.
\end{equation}
To complete the proof, it suffices to show that there exists some   $C>0$, depending only on $p$, such that
\begin{equation}\label{e202}
  \|u\|_{L^{p'}(U)}\leq C\|\nabla u\|_{L^{2}(U)}|U|^{1/p'}.
\end{equation}
To prove \eqref{e202}, we distinguish two cases:
\begin{itemize}
\item[(i)]
$1<p\leq 2.$  In this case,
\[1\leq s:= \frac{2p'}{p'+2}<2.\]
Then by the Sobolev embedding $W^{1,s}_{0}(\mathbb S^2_+)\hookrightarrow L^{p'}(\mathbb S^2_+)$ (note that $u^+\in W^{1,s}_{0}(\mathbb S^2_+)$ since $s<2$) and H\"older's inequality, 
 \[ 
  \|u\|_{L^{p'}(U)}=\|u^+\|_{L^{p'}(\mathbb S^2_+)}\leq C \|\nabla u^+\|_{L^{s}(\mathbb S^2_+)}= C\|\nabla u\|_{L^s(U)}\leq C \|\nabla u\|_{L^{2}(U)}|U|^{1/p'}.
 \] 
  Hence \eqref{e202} has been verified.

 \item[(ii)] $2<p<\infty.$ In this case,
 \[2< s:=\frac{2p'}{2-p'}<\infty.\]
By H\"older's inequality,
 \begin{equation}\label{e204}
\|u\|_{L^{p'}(U)}\leq \|u\|_{L^{2}(U)}\mathfrak |U|^{1/s}.
 \end{equation}
Using the Sobolev embedding $W^{1,1}_{0}(\mathbb S^2_+)\hookrightarrow L^{2}(\mathbb S^2_+),$  $\|u\|_{L^{2}(U)}$ can be estimated as follows:
 \begin{equation}\label{e205}
  \|u\|_{L^{2}(U)}=\|u^+\|_{L^{2}(\mathbb S^2_+)} \leq  C\|\nabla u^+\|_{L^{1}(\mathbb S^2_+)}= C\|\nabla u\|_{L^{1}(U)}\leq C\|\nabla u\|_{L^{2}(U)}|U|^{1/2}.
 \end{equation}
 From \eqref{e204} and \eqref{e205}, we get \eqref{e202}.
\end{itemize}
\end{proof}

 \section{Existence}\label{sec3}

 In this section, we give the proof of Theorem \ref{thme}.
 First we show the existence of a maximizer.
 
 \begin{proof}[Proof of Theorem \ref{thme}(i)]
 First, by Lemma \ref{greo} in Section \ref{sec2}, it is easy to check that $\mathcal E$ is bounded on $\mathcal R$, which implies that   $M\in\mathbb R.$

 Next we choose a sequence $\{v_n\}_{n=1}^\infty\subset\mathcal R$ such that
 \[\lim_{n\to\infty}\mathcal E(v_n)=M.\]
 Since $\{v_n\}$ is bounded in $L^p(\mathbb S^2_+)$, we can assume, up to a subsequence, that
 $v_n$ converges to some $ \hat v\in \overline{\mathcal R^w}$ weakly  in $L^p(\mathbb S^2_+)$.
 Here $\overline{\mathcal R^w}$ is the weak closure of $\mathcal R$ in $L^p(\mathbb S^2_+).$  Since $\mathcal E$ is weakly sequentially continuous in $L^p(\mathbb S^2_+)$, which can be easily verified by  Lemma \ref{greo}(ii), we have that
 \begin{equation}\label{htm01}
 \mathcal E(\hat v)=M.
 \end{equation}
  Denote by $L$ the derivative of $\mathcal E$ at $\hat v$, i.e.,
  \[L(v):=\int_{\mathbb S^2_+}v(\mathcal G_+\hat v-\lambda x_3)d\sigma.\]
  
  We claim that  
  \begin{equation}\label{clm01}
  \mbox{$\hat v$ is the unique maximizer of $L$ relative to $\overline{\mathcal R^w}.$}
   \end{equation}
   To this end, fix $ v\in \overline{\mathcal R^w}\setminus \{\hat v\}$. By \eqref{htm01} and the   fact that
 \[\sup_{  \overline{\mathcal R^w}}\mathcal E=\sup_{ \mathcal R}\mathcal E =M,\]
 we have that
 \begin{equation}\label{htm02}
 \mathcal E(\hat v)\geq \mathcal E(v).
 \end{equation}
On the other hand, using the symmetry of the Green operator $\mathcal G_+$ (cf. Lemma \ref{greo}(iii)),  the following identity holds:
 \begin{equation}\label{htm03}
 \mathcal E(v)-\mathcal E(\hat v)=L(v)-L(\hat v)+E(v-\hat v).
 \end{equation}
From \eqref{htm02} and \eqref{htm03},  we get
 \begin{equation}\label{lleq}
  L(\hat v)-L(v)\geq E(v-\hat v)>0.
  \end{equation}
Note that in the last inequality we used  Lemma \ref{greo}(iv). Hence \eqref{clm01} has been proved.
  
  Finally we show that 
  $\hat v\in\mathcal R.$ By Lemma \ref{lem201} in Section \ref{sec2}, there exists some $\tilde v\in {\mathcal R}$ such that 
  \[L(\tilde v)\geq L(v)\quad\forall\,v\in\mathcal R,\]
which, in combination with the fact that
\[ \sup_{v\in\mathcal R}L(v)=\sup_{v\in\bar{\mathcal R}}L(v),\]
implies that   $\tilde v$ is also  a maximizer of $L$ relative to $\overline{\mathcal R^w}$. Then by \eqref{clm01}, we get $\hat v=\tilde v\in\mathcal R$.
 \end{proof}

Next we prove Theorem \ref{thme}(ii).
 
\begin{proof}[Proof of Theorem \ref{thme}(ii)]
Repeating the argument in the proof of \eqref{clm01}, we see that any $v\in\mathcal M$ must be the unique maximizer of the following linear functional $L$
 relative to ${\mathcal R},$
 \[L(w):=\int_{\mathbb S^2_+}(\mathcal G_+v-\lambda x_3)w d\sigma. \]
  Then the desired assertion  is a straightforward consequence of Lemma \ref{lem202} in Section \ref{sec2}.
\end{proof}

Finally we prove Theorem \ref{thme}(iii). The main idea comes from  \cite{BAR} and \cite{CW20}.

\begin{proof}[Proof of Theorem \ref{thme}(iii)]

 Fix $v\in\mathcal M.$  Then by Theorem \ref{thme}(ii) there exists an increasing function $\phi:\mathbb R \to \mathbb R   \cup    \{\pm\infty\}$ such that
\begin{equation}\label{dd01}
v=\phi(\mathcal G_+v-\lambda x_3)\quad \mbox{a.e. on  }\mathbb S^2_+.
\end{equation}
  Denote
   \[C_\phi=\{s\in \mathbb R\mid \phi \mbox{ is continuous at } s\},\quad D_\phi=\mathbb R\setminus C_\phi.\]
Since $\phi$ is increasing,   $D_\phi$ must be countable. According to Lemma \ref{lem200} in Section \ref{sec2}, there exists a sequence  of  bounded, smooth real  functions $\{\phi_n\}$ such that
     \begin{equation}\label{gncr1}
     |\phi_n(s)|\leq |\phi(s)|\quad \forall\,s\in \mathbb R,\,n=1,2,\cdot\cdot\cdot,
     \end{equation}
 \begin{equation}\label{gncr2}
         \lim_{n\to\infty}\phi_n(s)=\phi(s)\quad \forall\,s\in C_\phi.
\end{equation}
Denote $v_n:=\phi_n(\mathcal G_+  v-\lambda x_3)$. Then by \eqref{gncr1},
 \begin{equation}\label{dd02}
      |v_n(\mathbf x)|\leq |v(\mathbf x)|\quad\mbox{a.e. }\mathbf x\in \mathbb S^2_+.
\end{equation}

We claim that
   \begin{equation}\label{aecv}
   \lim_{n\to \infty} v_n(\mathbf x)=  v(\mathbf x)\quad \mbox{a.e. }\mathbf x\in \mathbb S^2_+.
   \end{equation}
  In fact, if $\mathbf x\in (\mathcal G_+  v-\lambda x_3)^{-1}(C_{\phi})$, then  by \eqref{gncr2},
   \[\lim_{n\to\infty}\phi_n(\mathcal G_+  v-\lambda x_3)= \phi(\mathcal G_+  v-\lambda x_3).\]
Recalling \eqref{dd01}, we get
  \begin{equation}\label{ns09}
  \lim_{n\to\infty} v_n(\mathbf x)= v(\mathbf x) \quad\mbox{a.e. } \mathbf x\in (\mathcal G_+  v-\lambda x_3)^{-1}(C_{\phi}).
  \end{equation}
To deal with the case $\mathbf x\in (\mathcal G_+  v-\lambda x_3)^{-1}(D_{\phi})$, noting that for any $s\in D_\phi$,
\[ -\Delta (\mathcal G_+v-\lambda x_3)=0 \quad \mbox{ a.e. on }\,\, (\mathcal G_+v -\lambda x_3)^{-1}(s), \]
we obtain 
\[x_3=\frac{1}{2\lambda}\phi(s) \quad \mbox{ a.e. on }\,\, (\mathcal G_+v -\lambda x_3)^{-1}(s).\]
 Here we used the fact that 
 \begin{equation}\label{lx3}
 -\Delta x_3=2 x_3.
 \end{equation}
In other words,
\[ (\mathcal G_+v -\lambda x_3)^{-1}(s)\subset \left\{\mathbf x\in\mathbb S^2_+\mid x_3=\frac{1}{2\lambda}\phi(s)  \right\},\]
implying that 
 $(\mathcal G_+v -\lambda x_3)^{-1}(s)$ has measure zero.
Since  $D_{\phi}$ is countable, we deduce that
\begin{equation}\label{ns01}
\lim_{n\to \infty} v_n(\mathbf x)=   v(\mathbf x) \quad\mbox{ a.e. }\mathbf x\in (\mathcal G_+  v -\lambda x_3)^{-1}(D_\phi).
\end{equation}
From \eqref{ns09} and \eqref{ns01}, the claim \eqref{aecv} has been verified.

To proceed, notice that
\[{\rm div}\left(v_n\nabla^\perp  (\mathcal G_+v-\lambda x_3)\right)= \phi'_n(\mathcal G_+v-\lambda x_3)\nabla^\perp(\mathcal G_+ v-\lambda x_3)\cdot\nabla (\mathcal G_+ v-\lambda x_3)\equiv 0.\]
Hence for  any $\xi\in C_c^\infty(\mathbb S^2_+),$   integrating by parts gives
   \begin{equation}\label{vne0}
   \int_{\mathbb S^2_+}  v_n\nabla^\perp  (\mathcal G_+ v-\lambda x_3) \cdot\nabla \xi d\sigma=0.
   \end{equation}
Thanks to \eqref{dd02} and \eqref{aecv},  we can apply  Lebesgue's dominated convergence theorem to \eqref{vne0}  to get
\[ \int_{\mathbb S^2_+}  v\nabla^\perp  (\mathcal G_+ v-\lambda x_3) \cdot\nabla \xi d\sigma=0.\]
Thus the proof has been completed.
\end{proof}

 \section{Stability}\label{sec4}

The purpose of this section is to give the proof of Theorem \ref{thmos}.
We begin with the following compactness result.

\begin{proposition}[Compactness]\label{propcom}
Let $\{v_n\}_{n=1}^\infty\subset\mathcal R$ be a sequence such that
 \[\lim_{n\to\infty}\mathcal E(v_n)=M.\]
 Then, up to a subsequence, there exists  some $\hat v\in\mathcal M$ such that
 $v_{n}$ converges to  $\hat v$ strongly  in $L^p(\mathbb S^2_+)$ as $n\to\infty$.
\end{proposition}
\begin{proof}
   The main part of the proposition  has been verified in  Section \ref{sec3}. In fact, repeating the argument in the proof of Theorem \ref{thme}(i), we see that, up to a subsequence, there exists  some $\hat v\in\mathcal M$ such that
 $v_{n}$ converges to  $\hat v$ \emph{weakly}  in $L^p(\mathbb S^2_+)$ as $n\to\infty$. Since $\{v_n\}\subset\mathcal R$ and $\hat v\in\mathcal R,$ we have that
 \[\|v_n\|_{L^p(\mathbb S^2_+)}=\|v \|_{L^p(\mathbb S^2_+)}\quad\forall\,n.
\]
Strong convergence then follows from the uniform convexity of $L^p(\mathbb S^2_+)$ when $1<p<\infty$.
\end{proof}

 \begin{definition}[Admissible map]\label{damp3}
   A map $\eta:\mathbb R\mapsto L^p(\mathbb S^+)$ is called an admissible map if it satisfies
   \begin{itemize}
     \item [(i)]$\mathcal E(\eta(t))=\mathcal E(\eta(0))$ for any $t\in\mathbb R$;
     \item [(ii)] $\eta(t)\in\mathcal R_{\eta(0)}$ for any $t\in\mathbb R$.
   \end{itemize}
 \end{definition}

Theorem \ref{thmos} is a straightforward consequence of the following 
proposition.
\begin{proposition}
For any $\epsilon>0$, there exists some $\delta>0$, such that for any  admissible map  $\eta(t)$,   
 it holds that
  \[\inf_{v\in\mathcal M}\|\eta(0)-v\|_{L^p(\mathbb S^2_+)}<\delta\quad\Longrightarrow\quad \inf_{v\in\mathcal M}\|\eta(t)-v\|_{L^p(\mathbb S^2_+)}<\epsilon\quad\forall\,t\in\mathbb R.\]

\end{proposition}
\begin{proof} 
	It suffice to prove that for any sequence of admissible maps  $\{\eta_n\}_{n=1}^\infty$   and any sequence of times  $\{t_n\}_{n=1}^\infty \subset \mathbb R$, if
\begin{equation}\label{ems1}
 \lim_{n\to\infty}\|\eta_n(0)-\hat v\|_{L^p(\mathbb S^2_+)}= 0
\end{equation}
for some $ \hat v \in\mathcal M$, then, up to a subsequence,
\begin{equation}\label{ems2}
  \lim_{n\to\infty}\|\eta_n(t_n)-\tilde v\|_{L^p(\mathbb S^2_+)}= 0
\end{equation}
for some $ \tilde v\in\mathcal M$.

 By \eqref{ems1} and the continuity of $\mathcal E$ in $L^p(\mathbb S^2_+)$, 
  \begin{equation}\label{cvrm01}
  \lim_{n\to \infty}\mathcal E(\eta_n(0))=\mathcal E(\hat v)=M,
  \end{equation}
which implies that
  \begin{equation}\label{cvrm1}
\lim_{n\to \infty}\mathcal E(\eta_n(t_n)) =M.
  \end{equation}
Applying Lemma \ref{lem203} in Section \ref{sec2}, for each $n$, there exists some $\xi_n\in\mathcal R$ such that
\begin{equation}\label{infvw}
\|\xi_n-\eta_n(t_n)\|_{L^p(\mathbb S^2_+)}= \inf_{v\in\mathcal R,w\in\mathcal R_{\eta_n(t_n)}}\|v-w\|_{L^p(\mathbb S^2_+)}.
\end{equation}
Since  $\hat v\in\mathcal R$ and $\eta_n(0)\in\mathcal R_{\eta_n(t_n)}$ for each $n$, we deduce from \eqref{infvw} that
\begin{equation}\label{eta1}
\|\xi_n-\eta_n(t_n)\|_{L^p(\mathbb S^2_+)}\leq \|\eta_n(0)-\hat v\|_{L^p(\mathbb S^2_+)}.
\end{equation}
From \eqref{ems1} and \eqref{eta1}, we get
\begin{equation}\label{eta2}
\lim_{n\to \infty}\|\xi_n-\eta_n(t_n)\|_{L^p(\mathbb S^2_+)}=0.
\end{equation}
Taking into account \eqref{cvrm1},  we further get
\begin{equation}\label{eta3}
\lim_{n\to \infty}\mathcal E(\xi_n)=M.
\end{equation}

To summarize, we have constructed a  sequence $\{\xi_n\}\subset\mathcal R$ such that
\eqref{eta3} holds. Hence by Proposition \ref{propcom}, there is some $\tilde v\in\mathcal M$ such that, up to a subsequence,
\begin{equation}\label{eta4}
\lim_{n\to \infty}\|\xi_n-\tilde v\|_{L^p(\mathbb S^2_+)}=0.
\end{equation}
 Combining \eqref{eta2} and \eqref{eta4}, we get \eqref{ems2}. Hence the proof has been finished.

\end{proof}

  \section{Limiting behavior}\label{sec5}

 In this section, we give the proof of Theorem \ref{thmlb}.  Throughout this section, we will only be focusing on the case of $\varepsilon$ being sufficiently small, so we can assume, without loss of generality, that $\varepsilon\in(0, \pi/4]$.
 For convenience, we also use $C$ to denote various positive constants which may depend on $\lambda,\kappa,K$, but not on $\varepsilon$, \emph{whose specific values may change from line to line}.
 
 \subsection{Lagrange multiplier and vortex core}
 
For $v\in\mathcal M_\varepsilon,$ let $\phi_v$ be the increasing function in Theorem \ref{thme}(ii) such that \begin{equation}\label{l361}
 v= \phi_v(\mathcal G_+ v-\lambda x_3)  \quad \mbox{ a.e.  }\mathbf x\in\mathbb S^2_+.
\end{equation}
Define the Lagrange multiplier $\mu_v$ related to $v$ by setting
 \begin{equation}\label{dolm}
 \mu_v:=\inf\left\{s\in\mathbb R\mid \phi_v(s)>0\right\}.
 \end{equation}

 \begin{lemma}\label{sce01}
 For any  $v\in\mathcal M_\varepsilon,$ it holds that
\[
  \left\{\mathbf x\in\mathbb S^2_+\mid  v(\mathbf x)>0\right\}= \left\{\mathbf x\in\mathbb S^2_+\mid \mathcal G_+v(\mathbf x)-\lambda x_3>\mu_v\right\}
\]
 up to a set of measure zero. 
  \end{lemma}
  \begin{proof}
  From the definition of $\mu_v$ and the monotonicity of $\phi_v$, it is easy to see that
     \begin{equation}\label{muv01}
     v>0\quad\mbox{a.e. }\mathbf x\in \left\{\mathbf x\in\mathbb S^2_+\mid \mathcal G_+v(\mathbf x)-\lambda x_3>\mu_v\right\},
     \end{equation}
  and
   \begin{equation}\label{muv02}
   v=0\quad\mbox{a.e. }\mathbf x\in \left\{\mathbf x\in\mathbb S^2_+\mid \mathcal G_+v(\mathbf x)-\lambda x_3<\mu_v\right\}.
   \end{equation}
   By the property of Sobolev functions,  it holds that    \[-\Delta(\mathcal G_+v(\mathbf x)-\lambda x_3)=0 \quad \mbox{a.e. }\mathbf x\in   \left\{\mathbf x\in\mathbb S^2_+\mid \mathcal G_+v(\mathbf x)-\lambda x_3=\mu_v\right\},\]
 or equivalently,
    \begin{equation}\label{ls01}
     v(\mathbf x)=2\lambda x_3 \quad \mbox{a.e. }\mathbf x\in   \left\{\mathbf x\in\mathbb S^2_+\mid \mathcal G_+v(\mathbf x)-\lambda x_3=\mu_v\right\}.
        \end{equation}
       Here we used \eqref{lx3} again. On the other hand,  according to \eqref{l361},
    \begin{equation}\label{ls02}
    v(\mathbf x)=\phi_v(\mu_v) \quad \mbox{a.e. }\mathbf x\in   \left\{\mathbf x\in\mathbb S^2_+\mid \mathcal G_+v(\mathbf x)-\lambda x_3=\mu_v\right\}.
    \end{equation}
Combining \eqref{ls01} and \eqref{ls02}, we deduce that
    \begin{align*}
    \left\{\mathbf x\in\mathbb S^2_+\mid \mathcal G_+v(\mathbf x)-\lambda x_3=\mu_v\right\} &\subset   \left\{\mathbf x\in\mathbb S^2_+\mid  v(\mathbf x)=2\lambda x_3\right\}\cap  \left\{\mathbf x\in\mathbb S^2_+\mid  v(\mathbf x)=\phi_v(\mu_v)\right\}\\
    &\subset  \left\{\mathbf x\in\mathbb S^2_+\mid  2\lambda x_3=\phi_v(\mu_v)\right\}
    \end{align*}
    up to a set of measure zero.
Since
\[\left|\left\{\mathbf x\in\mathbb S^2_+\mid  2\lambda x_3=\phi_v(\mu_v)\right\}\right|=0,\]
we get
    \begin{equation}\label{muv03}
   \left |\left\{\mathbf x\in\mathbb S^2_+\mid \mathcal G_+v(\mathbf x)-\lambda x_3=\mu_v\right\}\right|=0.
    \end{equation}
    The desired assertion then follows from \eqref{muv01}, \eqref{muv02} and \eqref{muv03}.
  \end{proof}

 \begin{proof}[Proof of Theorem \ref{thmlb}(i)]
For any  $v\in\mathcal M_\varepsilon$, define  
 \[O_v:=\left\{\mathbf x\in\mathbb S^2_+\mid \mathcal G_+v(\mathbf x)-\lambda x_3>\mu_v\right\}.\]
Since $\mathcal G_+v  -\lambda x_3$ is continuous by elliptic regularity and Sobolev embedding, we see that  $O_v$
 is an open set. 
Then Theorem \ref{thmlb}(i) is a straightforward consequence of Lemma \ref{sce01}.
 \end{proof}

 \subsection{Lower bound of Lagrange multiplier}

For $v\in\mathcal M_\varepsilon,$ recall the definition \eqref{dolm} of the Lagrange multiplier $\mu_v$ related to $v$. The aim of this subsection is to prove the following proposition.

 \begin{proposition}\label{lvlm}
 $\mu_v>-\frac{\kappa}{2\pi}\ln\varepsilon-C$ for any $v\in\mathcal M_\varepsilon.$
 \end{proposition}

  To prove Proposition \ref{lvlm}, we begin with the following lower bound for the kinetic energy of the maximizers.

 \begin{lemma}\label{lbke}
 $E(v)>-\frac{\kappa^2}{4\pi}\ln\varepsilon-C$ for any $v\in\mathcal M_\varepsilon.$
 \end{lemma}
 
 \begin{proof}
  Since $\varrho_\varepsilon\in\mathcal R_\varepsilon,$ we have that
  \[E(v)-\lambda I(v)\geq E(\varrho_\varepsilon)-\lambda I(\varrho_\varepsilon).\] 
  It is easy to check that 
  \begin{equation}\label{efi1}
  \sup_{w\in\mathcal R_\varepsilon}|I(w)|< C.
  \end{equation}
Thus, to complete the proof, it suffices to estimate the lower bound of $E(\varrho_\varepsilon).$ We estimate as follows:
   \begin{align*}
 E(\varrho_\varepsilon)&=\frac{1}{4\pi}\int_{\mathbb S^2_+} \int_{\mathbb S^2_+} \ln\frac{|\mathbf x-\tilde{\mathbf y}|}{|\mathbf x-\mathbf y|}\varrho_\varepsilon(\mathbf x)\varrho_\varepsilon(\mathbf y)d\sigma_{\mathbf x}d\sigma_{\mathbf y}\\
 & =\frac{1}{4\pi}\int_{B_\varepsilon(N)} \int_{B_\varepsilon(N)} \ln\frac{|\mathbf x-\tilde{\mathbf y}|}{|\mathbf x-\mathbf y|}\varrho_\varepsilon(\mathbf x)\varrho_\varepsilon(\mathbf y)d\sigma_{\mathbf x}d\sigma_{\mathbf y}\\
 &\geq  \frac{\kappa^2}{4\pi}\ln\frac{\cos\varepsilon}{\sin\varepsilon}
 \\
  &>-\frac{\kappa^2}{4\pi}\ln \varepsilon -C.
   \end{align*}
  Note that in the first inequality  we have used the following geometric facts:
   \[|\mathbf x-\mathbf y|\leq 2\sin\varepsilon,\quad |\mathbf x-\tilde{\mathbf y}|\geq 2\cos\varepsilon\quad \forall\,\mathbf x,\mathbf y\in B_\varepsilon(N).\]
 \end{proof}

 \begin{lemma}\label{mlb01}
 $\mu_v> -\lambda$ for any $v\in\mathcal M_\varepsilon.$
 \end{lemma}
\begin{proof}
Suppose by contradiction that $\mu_v\leq -\lambda$.
 Then for any $\mathbf x\in\mathbb S^2_+,$ it holds that
   \[\mathcal G_+v(\mathbf x)-\lambda x_3\geq \mathcal G_+v(\mathbf x)-\lambda  >-\lambda \geq \mu_v,\]
   which in combination with Lemma \ref{sce01} implies that \begin{equation}\label{sub1}
   \mathbb S^2_+\subset O_v.
   \end{equation}
   On the other hand, since $v\in\mathcal R_\varepsilon,$
   we have that
   \[|O_v|=\left|\left\{\mathbf x\in\mathbb S^2_+\mid v(\mathbf x)>0\right\}\right|=\left|\left\{\mathbf x\in\mathbb S^2_+\mid \varrho_\varepsilon(\mathbf x)>0\right\}\right|=|B_\varepsilon(N)|<|\mathbb S^2_+|,\]
   a contradiction to \eqref{sub1}.
 \end{proof}

 \begin{lemma}\label{ubkv}
 It holds that
 \[\int_{\mathbb S^2_+}v(\mathcal G_+v-\lambda x_3-\mu_v-\lambda)d\sigma\leq  C \quad\forall\,v\in\mathcal M_\varepsilon.\]
 \end{lemma}
\begin{proof}
 Denote $u:=\mathcal G_+v-\lambda x_3-\mu_v-\lambda$. Then $u\in W^{2,p}(\mathbb S^2_+)$ and
 \[-\Delta u=v-2\lambda x_3,\]
 where we have used \eqref{lx3}.
  Moreover, in view of Lemma \ref{mlb01},  it holds that 
 \[u<0\,\, \mbox{ on }\,\,\partial \mathbb S^2_+.\]
 Denote 
 \[U:=\left\{\mathbf x\in\mathbb S^2_+\mid u(\mathbf x)>0\right\}.\]
 Then it is clear that 
  \begin{equation}\label{sub02}
  U\subset \left\{\mathbf x\in\mathbb S^2_+\mid \mathcal G_+v(\mathbf x)-\lambda x_3>\mu_v \right\}=O_v.
  \end{equation}
By integration by parts,
\begin{align*}
\int_{\mathbb S^2_+}uvd\sigma&\leq\int_{\mathbb S^2_+}u^+vd\sigma \\
&=\int_{\mathbb S^2_+}u^+(-\Delta u+2\lambda x_3)d\sigma  \\
&=\int_{\mathbb S^2_+}\nabla u\cdot\nabla u^+ +2\lambda  x_3u^+ d\sigma  \\
&=\int_{\mathbb S^2_+}|\nabla u^+|^2d\sigma+2\lambda \int_{\mathbb S^2_+} x_3u^+d\sigma\\
&\leq \int_{\mathbb S^2_+}|\nabla u^+|^2d\sigma+ C \int_{\mathbb S^2_+} |u^+|^2d\sigma\\
&\leq  \|\nabla u^+\|^2_{L^2(\mathbb S^2_+)}+C\|\nabla u^+\|_{L^2(\mathbb S^2_+)}\\
&=  \|\nabla u\|^2_{L^2(U)}+C\|\nabla u\|_{L^2(U)}.
\end{align*}
Note that we have used the Poincar\'e inequality in the last inequality.
On the other hand, using  Lemma \ref{tcl}, we can estimate $ \|\nabla u\|_{L^2(U)}$ as follows:
   \begin{align*}
    \int_{U}|\nabla u|^2d\sigma     \leq  &C \|\Delta u \|_{L^{p}(U)}  |U|^{1/p'}\\
     \leq &C\|v-2\lambda x_3\|_{L^{p}(\mathbb S^2_+)} |O_v|^{1/p'}  \\
     \leq& C|O_v|^{1/p'} \left(\|v\|_{L^{p}(\mathbb S^2_+)}+2\lambda \|x_3\|_{L^{p}(\mathbb S^2_+)}\right) \\
    \leq & C\varepsilon^{2/p'} \left(\|\varrho_\varepsilon\|_{L^{p}(\mathbb S^2_+)}+C\right) \\
     \leq  & C,
   \end{align*}
  where we have used  \eqref{sub02} and (H3). Thus the proof has been completed.

\end{proof}

Having proved Lemmas \ref{lbke} and \ref{ubkv},  we are ready to prove Proposition \ref{lvlm}.
\begin{proof}[Proof of Proposition \ref{lvlm}]
Notice that
\[\int_{\mathbb S^2_+}v(\mathcal G_+v-\lambda x_3-\mu_v-\lambda)d\sigma=2E(v)-\lambda I(v)-\kappa \mu_v-\kappa\lambda.\]
  Taking into account \eqref{efi1} and Lemma \ref{ubkv}, we deduce that
  \[ \kappa\mu_v\geq  2E(v)-C.\]
Then the desired estimate for $\mu_v$ is an easy consequence of  Lemma \ref{lbke}.
\end{proof}

   \subsection{Size of vortex core}

In this subsection, we give the proof of Theorem \ref{thmlb}(ii). We begin with the following estimate on the diameter of the vortex core.

 \begin{lemma}\label{sovc}
For any $v\in\mathcal M_\varepsilon$, it holds that
\begin{equation}\label{sovc1}
d(\mathbf x,\mathbf y)\leq C\varepsilon\quad \forall\,\mathbf x,\mathbf y\in O_v,
\end{equation}
where $d(\mathbf x,\mathbf y)$ the geodesic distance 
between $\mathbf x$ and $\mathbf y$ (cf. \eqref{gds}).
  \end{lemma}
 \begin{proof}
 Fix $\mathbf x\in  O_v$.  Then by Proposition \ref{lvlm},
 \[\mathcal G_+v(\mathbf x)-\lambda x_3>\mu_v>-\frac{\kappa}{2\pi}\ln\varepsilon-C,\]
which can be written as
\begin{equation}\label{gmc0}
 \int_{\mathbb S^2_+}\ln\frac{\varepsilon}{|\mathbf x-\mathbf y|}v(\mathbf y)d\sigma >-\int_{\mathbb S^2_+}\ln |\mathbf x-\mathbf y'|v(\mathbf y) d\sigma -C.
\end{equation}
  Since $|\mathbf x-\mathbf y'|\leq 2$ for any $\mathbf y\in\mathbb S^2_+,$ we get from \eqref{gmc0} that
  \begin{equation}\label{gmc1}
  \int_{\mathbb S^2_+}\ln\frac{\varepsilon}{|\mathbf x-\mathbf y|}v(\mathbf y)d\sigma >-C.
  \end{equation}
Write
 \begin{equation}\label{gmc2}
  \int_{\mathbb S^2_+}\ln\frac{\varepsilon}{|\mathbf x-\mathbf y|}v(\mathbf y)d\sigma=\int_{\mathbf y\in\mathbb S^2_+,\,|\mathbf x-\mathbf y|\geq  R\varepsilon}\ln\frac{\varepsilon}{|\mathbf x-\mathbf y|}v(\mathbf y)d\sigma+\int_{\mathbf y\in\mathbb S^2_+,\,|\mathbf x-\mathbf y|<  R\varepsilon}\ln\frac{\varepsilon}{|\mathbf x-\mathbf y|}v(\mathbf y)d\sigma,
  \end{equation}
 where $R>1$ will be determined later.
 For the first term on the right-hand side of \eqref{gmc2}, we have that
  \begin{equation}\label{gmc3}
  \int_{\mathbf y\in\mathbb S^2_+,\,|\mathbf x-\mathbf y|\geq   R\varepsilon}\ln\frac{\varepsilon}{|\mathbf x-\mathbf y|}v(\mathbf y)d\sigma \leq -\ln R\int_{\mathbf y\in\mathbb S^2_+,\,|\mathbf x-\mathbf y|\geq  R\varepsilon} v(\mathbf y)d\sigma.
  \end{equation}
 For the second term on the right-hand side of \eqref{gmc2}, by H\"older's inequality,
 \begin{align*}
 \int_{\mathbf y\in\mathbb S^2_+,\,|\mathbf x-\mathbf y|<  R\varepsilon}\ln\frac{\varepsilon}{|\mathbf x-\mathbf y|}v(\mathbf y)d\sigma &\leq \int_{\mathbf y\in\mathbb S^2_+,\,|\mathbf x-\mathbf y|< \varepsilon}\ln\frac{\varepsilon}{|\mathbf x-\mathbf y|}v(\mathbf y)d\sigma \\
 &\leq \|v\|_{L^p(\mathbb S^2_+)}\left(\int_{\mathbf y\in\mathbb S^2,\,|\mathbf x-\mathbf y|<\varepsilon}\bigg| \ln\frac{\varepsilon}{|\mathbf x-\mathbf y|}\bigg|^{p'}d\sigma\right)^{1/p'}\\
 &= \|v\|_{L^p(\mathbb S^2_+)}\bigg\| \ln\frac{\varepsilon}{|\mathbf x-\mathbf y|}\bigg\|_{L^{p'}(B_{2\arcsin(\varepsilon/2)}(\mathbf x))}.
 \end{align*}
The logarithmic integral involved above can be computed explicitly using spherical coordinates:
 \begin{align*}
 \int_{B_{2\arcsin(\varepsilon/2)}(\mathbf x) }\bigg| \ln\frac{\varepsilon}{|\mathbf x-\mathbf y|}\bigg|^{p'}d\sigma&=\int_{B_{2\arcsin(\varepsilon/2)}(N) }\bigg| \ln\frac{\varepsilon}{|\mathbf y-N|}\bigg|^{p'}d\sigma
 \\
 &=2\pi\int_{\pi/2-2\arcsin(\varepsilon/2)}^{ \pi/2}
 \left|\ln\frac{\varepsilon}{2\sin\left(\frac{\pi}{4}-\frac{\theta}{2}\right)}\right|^{p'}\cos\theta d\theta\\
 &=2\pi\int_{0}^{ 2\arcsin(\varepsilon/2)}
 \left|\ln\frac{\varepsilon}{2\sin\left(\frac{\vartheta}{2}\right)}\right|^{p'}\sin\vartheta d\vartheta\quad\quad\left(\vartheta:=\frac{\pi}{2}-\theta\right)\\
 &=2\pi\int_0^\varepsilon \left|\ln\frac{\varepsilon}{ s}\right|^{p'}sds\quad\quad\left(s:=2\sin\left(\frac{\vartheta}{2}\right)\right)\\
 &=2\pi\varepsilon^2 \int_0^1|\ln s|^{p'}sds. 
 \end{align*}
 Hence the second term is uniformly bounded from above:
 \begin{equation}\label{gmc4}
 \int_{\mathbf y\in\mathbb S^2_+,\,|\mathbf x-\mathbf y|< R\varepsilon}\ln\frac{\varepsilon}{|\mathbf x-\mathbf y|}v(\mathbf y)d\sigma \leq C\varepsilon^{2/{p'}}\|v\|_{L^p(\mathbb S^2_+)}\leq C.
 \end{equation}
From \eqref{gmc1}-\eqref{gmc4}, we get
 \[\int_{\mathbf y\in\mathbb S^2_+,\,|\mathbf x-\mathbf y|\geq   R\varepsilon} v(\mathbf y)d\sigma \leq \frac{C}{\ln R}.\]
 Choosing $R$ large enough such that 
 \[\frac{C}{\ln R}<\frac{1}{3}\kappa,
 \] 
 and using the fact that $\int_{\mathbb S^2_+}v d\sigma=\kappa$, we get
  \begin{equation}\label{toth}
  \int_{\mathbf y\in\mathbb S^2_+,\,|\mathbf x-\mathbf y|<  R\varepsilon} v(\mathbf y)d\sigma >\frac{2}{3}\kappa.
  \end{equation}
  Since \eqref{toth} holds for arbitrary $\mathbf x\in O_v,$ we get
   \begin{equation}\label{arsin}
   d(\mathbf x,\mathbf y)\leq 4\arcsin\left(\frac{R\varepsilon}{2}\right)\quad \forall\,\mathbf x,\mathbf y\in O_v.
   \end{equation}
   In fact, if \eqref{arsin} is false, then we can take two points $\mathbf x_1,\mathbf x_2\in O_v$ such that 
   \[\{\mathbf y\in\mathbb S^2_+\mid |\mathbf x_1-\mathbf y|< R\varepsilon\}\cap\{\mathbf y\in\mathbb S^2_+\mid |\mathbf x_2-\mathbf y|< R\varepsilon\}=\varnothing. \]
   Then  by \eqref{toth},
  \[ \int_{\mathbb S^2_+}v d\sigma\geq  \int_{\mathbf y\in\mathbb S^2_+,\,|\mathbf x_1-\mathbf y|<  R\varepsilon}v(\mathbf y) d\sigma + \int_{\mathbf y\in\mathbb S^2_+,\,|\mathbf x_2-\mathbf y|<  R\varepsilon}v(\mathbf y) d\sigma >\frac{4}{3}\kappa\]
a contradiction to  $\int_{\mathbb S^2_+}v d\sigma=\kappa$.
The desired estimate \eqref{sovc1} then  follows from \eqref{arsin}.
 \end{proof}

 Recall the definition of mass center $\mathbf X_v$ of $v\in\mathcal M_\varepsilon$, i.e.,
   \[\mathbf X_v:=\frac{1}{\kappa}\int_{\mathbb S^2_+}\mathbf xv d\sigma.\]

 \begin{proof}[Proof of Theorem \ref{thmlb}(ii)]
 It suffices to notice that for any $\mathbf x\in O_v,$  
 \[\left|\mathbf x-\mathbf X_v\right|\leq \frac{1}{\kappa}\int_{\mathbb S^2_+}|\mathbf x-\mathbf y|v(\mathbf y)d\sigma 
 =\frac{1}{\kappa}\int_{O_v}|\mathbf x-\mathbf y|v(\mathbf y)d\sigma \leq  C\varepsilon,\]
where we have used Lemma \ref{sovc} in the last inequality.
 \end{proof}

\subsection{Limiting location of vortex cores}

In this subsection, we give the proof of Theorem \ref{thmlb}(iii). Let $v_\varepsilon\in\mathcal M_\varepsilon$ be arbitrarily chosen. It is clear that $\mathbf X_{v_\varepsilon}\in\mathbb S^2_+,$  where 
 $\mathbf X_{ v_\varepsilon}$ is the mass center of ${ v_\varepsilon}$ (cf. \eqref{doms}).
 So, 
  up to a subsequence, we can assume that  
\[\lim_{\varepsilon\to 0}\mathbf X_{v_\varepsilon}=\hat{\mathbf x}\in \left\{ (x_1,x_2,x_3)\in\mathbb S^2\mid x_3\geq 0 \right\}.\]
Our aim is to show that the third component $\hat x_3$ of $\hat{\mathbf x}$ satisfies
  \begin{equation}\label{tdcz}
  \hat x_3=\begin{cases}
 \frac{\kappa}{4\pi\lambda},&\mbox{if }\,\,\frac{\kappa}{4\pi\lambda}\leq  1,\\
 1,&\mbox{if }\,\,\frac{\kappa}{4\pi\lambda}>1.
 \end{cases}
 \end{equation}
For convenience,  we extend $v_\varepsilon$ to the whole sphere by setting $v_\varepsilon\equiv 0$ on $\mathbb S^2_-.$

Fix an arbitrary point $\tilde {\mathbf x}\in\mathbb S^2_+.$ 
For sufficiently small $\varepsilon,$  in view of Theorem \ref{thmlb}(ii), we can always take some $\mathsf R_\varepsilon\in\mathbf S\mathbf O(3)$ such that $\tilde v_\varepsilon:=v_\varepsilon\circ \mathsf R_\varepsilon$ satisfies
\[\tilde v_\varepsilon\in\mathcal R_\varepsilon\quad\mbox{and}\quad  \mathbf X_{\tilde v_\varepsilon}=\tilde {\mathbf x}.\] 
Then
 $\mathcal E(v_\varepsilon)\geq \mathcal E(\tilde v_\varepsilon),$ i.e.,
\begin{equation}\label{vve1}
\begin{split}
 &\frac{1}{4\pi}\int_{\mathbb S^2_+}\int_{\mathbb S^2_+}\ln\frac{|\mathbf x-\mathbf y'|}{|\mathbf x-\mathbf y|}v_{\varepsilon}(\mathbf x)v_\varepsilon(\mathbf y)d\sigma_{\mathbf x}d\sigma_{\mathbf y}-\lambda \int_{\mathbb S^2_+}x_3v_{\varepsilon}(\mathbf x) d\sigma \\
 \geq &\frac{1}{4\pi}\int_{\mathbb S^2_+}\int_{\mathbb S^2_+}\ln\frac{|\mathbf x-\mathbf y'|}{|\mathbf x-\mathbf y|}\tilde v_{\varepsilon}(\mathbf x)\tilde v_\varepsilon(\mathbf y)d\sigma_{\mathbf x}d\sigma_{\mathbf y}-\lambda \int_{\mathbb S^2_+}x_3\tilde v_{\varepsilon}(\mathbf x) d\sigma.
 \end{split}
\end{equation}
On the other hand, 
\begin{equation}\label{vve2}
\begin{split}
\int_{\mathbb S^2_+} \int_{\mathbb S^2_+}\ln {|\mathbf x-\mathbf y|}\tilde v_{\varepsilon}(\mathbf x)\tilde v_\varepsilon(\mathbf y)d\sigma_{\mathbf x}d\sigma_{\mathbf y}&=\int_{\mathsf R_\varepsilon^{-1}(O_{v_\varepsilon})} \int_{\mathsf R_\varepsilon^{-1}(O_{v_\varepsilon})}\ln {|\mathbf x-\mathbf y|}  \tilde v_{\varepsilon} ( \mathbf x ) \tilde v_\varepsilon (\mathbf y)d\sigma_{\mathbf x}d\sigma_{\mathbf y}\\
&=\int_{ O_{v_\varepsilon}} \int_{O_{v_\varepsilon}}\ln {\left|\mathsf R_\varepsilon^{-1}\mathbf x- \mathsf R_\varepsilon^{-1}\mathbf y\right|}  v_{\varepsilon}(\mathbf x)  v_\varepsilon( \mathbf y)d\sigma_{\mathbf x}d\sigma_{\mathbf y}\\
&= \int_{\mathbb S^2_+}\int_{\mathbb S^2_+}\ln {|\mathbf x-\mathbf y|}v_{\varepsilon}(\mathbf x)v_\varepsilon(\mathbf y)d\sigma_{\mathbf x}d\sigma_{\mathbf y}.
\end{split}
\end{equation}
Note that in the first equality of \eqref{vve2} we have used 
\[\left\{\mathbf x\in\mathbb S^2_+\mid \tilde v_\varepsilon(\mathbf x)>0\right\}=\mathsf R^{-1}_\varepsilon(O_{v_\varepsilon}),\]
as can be readily observed from the definition of $\tilde v_\varepsilon.$
From \eqref{vve1} and \eqref{vve2}, we get
\begin{equation}\label{vve3}
\begin{split}
 &\frac{1}{4\pi}\int_{\mathbb S^2_+}\int_{\mathbb S^2_+}\ln |\mathbf x-\mathbf y'| v_{\varepsilon}(\mathbf x)v_\varepsilon(\mathbf y)d\sigma_{\mathbf x}d\sigma_{\mathbf y}-\lambda \int_{\mathbb S^2_+}x_3v_{\varepsilon}(\mathbf x) d\sigma  \\
 \geq &\frac{1}{4\pi}\int_{\mathbb S^2_+}\int_{\mathbb S^2_+}\ln |\mathbf x-\mathbf y'| \tilde v_{\varepsilon}(\mathbf x)\tilde v_\varepsilon(\mathbf y)d\sigma_{\mathbf x}d\sigma_{\mathbf y}-\lambda \int_{\mathbb S^2_+}x_3\tilde v_{\varepsilon}(\mathbf x) d\sigma.
 \end{split}
\end{equation}
Using Theorem \ref{thmlb}(ii) again, we can pass to the limit $\varepsilon\to0$ in \eqref{vve3} to get $\hat x_3>0$ and
\[\frac{\kappa^2 }{4\pi}\ln(2\hat x_3)- \kappa\lambda \hat x_3\geq \frac{\kappa^2 }{4\pi}\ln(2\tilde x_3)-\kappa\lambda \tilde x_3.\]
Since $\tilde {\mathbf x}\in\mathbb S^2_+$ is arbitrary, we deduce that   $\hat x_3$ is in fact a global maximum point of the function
  \[W(s):=\frac{\kappa}{4\pi}\ln s-\lambda s \]
 in $(0,1]$. On the other hand, as can be readily verified, the function $W$ admits a unique global maximum point $s^*$ in $(0,1]$, given by
 \[s^*=   \begin{cases}
 \frac{\kappa}{4\pi\lambda},&\mbox{if }\,\,\frac{\kappa}{4\pi\lambda}\leq 1,\\
 1,&\mbox{if }\,\,\frac{\kappa}{4\pi\lambda}>1.
 \end{cases}\]
This proves Theorem \ref{thmlb}(iii).

 \section{Rotating sphere}\label{sec6}

 If the sphere rotates around the polar axis $\mathbf e_3$ at a constant angular speed $\Omega$, then any fluid particle at $\mathbf x=(x_1,x_2,x_3)\in\mathbb S^2$ with velocity $\mathbf v\in T_{\mathbf x}\mathbb S^2$ is affected by the Coriolis force $2\Omega x_3 J\mathbf v$. In this case, the Euler equation becomes
 \begin{equation}\label{eurs}
  \begin{cases}
    \partial_t\mathbf v+\nabla_{\mathbf v}\mathbf v=2\Omega x_3J\mathbf v -\nabla P, & t\in\mathbb R,\,\,  \mathbf x\in\mathbb S^2,\\
    {\rm div}\, \mathbf v=0.
  \end{cases}
\end{equation}

Repeating the argument in Subsection \ref{sec11}, we have the following vorticity formulation of \eqref{eurs}:
\[\partial_t\omega+J\nabla \mathcal G\omega \cdot\nabla(\omega+2\Omega x_3)=0,\quad t\in\mathbb R,\,\,  \mathbf x\in\mathbb S^2.\]
Introduce the absolute vorticity  
\[ \zeta:=\omega+2\Omega x_3,\]
which measures  the rotation of fluid particles relative to the sphere, taking into account both the rotation from the internal motion of the fluid and the rotation of the sphere.
 In terms of $\zeta,$ the  Euler equation  can be written as
\[\partial_t\zeta+  J\nabla(\mathcal G\zeta -\Omega x_3)\cdot\nabla \zeta=0,\quad t\in\mathbb R,\,\,  \mathbf x\in\mathbb S^2. \tag{$E_\Omega$} \]
Note that $(E_0)$ is exactly \eqref{vor0} with $\zeta=\omega$.
A useful relation between the solutions of $(E_0)$ and those of $(E_\Omega)$ is that
\begin{equation}\label{v0v}
 \mbox{\emph{$\zeta(t,\mathbf x)$ is a solution of $(E_0)$ if and only if  $\zeta(t,\mathsf R^{\mathbf e_3}_{\Omega t}\mathbf x)$ is a solution of $(E_\Omega)$.}}
\end{equation}

In the odd-symmetric setting, $(E_\Omega)$ is equivalent to
 \begin{equation}\label{eunr}
 \begin{cases}
 \partial_t\zeta +J\nabla(\mathcal G_+\zeta-\Omega x_3)\cdot\nabla \zeta=0, &t\in\mathbb R,\,\,\mathbf x\in \mathbb S^2_+,\\
 \mathcal G_+\zeta(t,\mathbf x)=  \frac{1}{2\pi}\int_{\mathbb S^2}\ln\frac{|\mathbf x-\mathbf y'|}{|\mathbf x-\mathbf y|}\zeta(t,\mathbf y)d\sigma.
 \end{cases}
 \end{equation}
Using \eqref{v0v}, it is easy to see that for any sufficiently regular solution $\zeta$ of \eqref{eunr}, the conservation laws \eqref{cl01}-\eqref{cl03} still holds with $\omega$ being replaced by $\zeta$.

For the rotating case, the motion of $N$ point vortices  $\mathbf x_1,\cdot\cdot\cdot,\mathbf x_N$ are described by the following system of ODEs:
 \begin{equation}\label{odes1r}
 \frac{d\mathbf x_i }{dt}=\sum_{j\neq i}\kappa_jJ\nabla_\mathbf x G(\mathbf x,\mathbf x_j)\bigg|_{\mathbf x=\mathbf x_i}-\Omega J\nabla  x_3,
 \end{equation}
which can also be written as
 \begin{equation}\label{odes2r}
 \frac{d\mathbf x_i }{dt}=\sum_{j\neq i}\frac{\kappa_j}{2\pi}\frac{\mathbf x_j\times\mathbf x_i}{|\mathbf x_i-\mathbf x_j|^2}-\Omega \mathbf e_3\times \mathbf x_i.
 \end{equation}
If $N=2,$ then \eqref{odes2r} becomes
  \begin{equation}\label{odes3r}
  \begin{cases}
 \frac{d\mathbf x_1 }{dt}= \frac{\kappa_2}{2\pi}\frac{\mathbf x_2\times\mathbf x_1}{|\mathbf x_1-\mathbf x_2|^2}-\Omega \mathbf e_3\times \mathbf x_1,\\
  \frac{d\mathbf x_2 }{dt}= \frac{\kappa_1}{2\pi}\frac{\mathbf x_1\times\mathbf x_2}{|\mathbf x_1-\mathbf x_2|^2}-\Omega \mathbf e_3\times \mathbf x_2.
 \end{cases}
 \end{equation}
For a pair of odd-symmetric vortices with $\kappa_1=-\kappa_2=\kappa$, an explicit rotating solution is
 \begin{equation}\label{2vmr}
 \begin{cases}
 \mathbf x_1(t)=\left(\cos\theta_0\cos ((\lambda -\Omega) t+\varphi_0),\cos\theta_0\sin( (\lambda -\Omega)   t+\varphi_0),\sin\theta_0\right),\\
  \mathbf x_2(t)=\left(\cos\theta_0\cos((\lambda -\Omega)  t+\varphi_0),
\cos\theta_0\sin((\lambda -\Omega)  t+\varphi_0),-\sin\theta_0\right),
  \end{cases}
  \end{equation}
where $\theta_0,\varphi_0\in\mathbb R$, and $\lambda$ still satisfies the relation  
\[4\pi\lambda\sin\theta_0=\kappa.\]

Based on the above discussion, especially using the relation \eqref{v0v}, we obtain the following theorem.

 \begin{theorem}
Let $\mathcal M$ be the set of maximizers of $(V)$ in Theorem \ref{thme}.
  Then
   \begin{itemize}
   \item [(i)]  For any $v\in\mathcal M$,  $v(\mathsf R^{\mathbf e_3}_{ (\Omega-\lambda) t}\mathbf x)$ solves \eqref{eunr} in the weak sense.
   \item[(ii)] The set $\mathcal M$ is stable under the dynamics of \eqref{eunr} in the following sense: for any $\epsilon>0$, there exists some $\delta>0$, such that for any  sufficiently smooth solution $\zeta(t,\mathbf x)$ of \eqref{eunr}, it holds that
  \[\inf_{v\in\mathcal M}\|\zeta(0,\cdot)-v\|_{L^p(\mathbb S^2_+)}<\delta\quad\Longrightarrow\quad \inf_{v\in\mathcal M}\|\zeta(t,\cdot)-v\|_{L^p(\mathbb S^2_+)}<\epsilon\quad\forall\,t\in\mathbb R.\]
       \end{itemize}
\end{theorem}

\bigskip
\noindent{\bf Acknowledgements:} D. Cao was supported by National Key R\&D Program of China
	(Grant 2022YFA1005602) and NNSF of China (Grant 12371212). S. Li and G. Wang were supported by NNSF of China (Grant 12471101).

\bigskip
\noindent{\bf  Data availability statement} All data generated or analyzed during this study are included in this published article.

\bigskip
\noindent{\bf Conflict of interest}  The authors declare that they have no conflict of interest to this work.

\phantom{s}
 \thispagestyle{empty}

\end{document}